\documentclass[notitlepage, 11pt]{article}

\usepackage{color}
\usepackage{amsmath}
\usepackage{latexsym,amssymb}
\usepackage{graphicx}
\usepackage{stmaryrd}
\usepackage{fixmath}

\newtheorem{lemma}{Lemma}
\newtheorem{theorem}{Theorem}

\newtheorem{problem}{Problem}

\newcommand{\noi}{\noindent}

\newcommand{\R}{\mathbb{R}}
\newcommand{\N}{\mathbb{N}}

\newcommand{\la}{\lambda}
\newcommand{\sig}{\sigma}

\newcommand{\eps}{\varepsilon}
\newcommand{\ph}{\varphi}

\newcommand{\al}{\alpha}

\newcommand{\io}{\iota}

\newcommand{\del}{\delta}

\newcommand{\Gam}{\mathnormal{\Gamma}}

\newcommand{\Sig}{\mathnormal{\Sigma}}

\newcommand{\Ps}{\mathnormal{\Psi}}

\newcommand{\Om}{\mathnormal{\Omega}}

\newcommand{\C}{{\mathbb C}}
\newcommand{\D}{{\mathbb D}}

\newcommand{\Z}{{\mathbb Z}}

\newcommand{\EE}{{\mathbb E}}

\newcommand{\PP}{{\mathbb P}}

\newcommand{\bdel}{\mathbold{\delta}}

\newcommand{\calF}{{\cal F}}

\newcommand{\calS}{{\mathcal{S}}}

\newcommand{\be}{{\mathbf e}}

\newcommand{\skp}{\vspace{\baselineskip}}

\newcommand{\diag}{{\rm diag}}

\newcommand{\w}{\wedge}

\newcommand{\To}{\Rightarrow}

\newcommand{\iy}{\infty}

\newcommand{\qed}{\hfill $\Box$}

\begin{document}

\title{A note on non-existence of diffusion limits for
serve-the-longest-queue when the buffers are equal
in size}

\author{Rami Atar\thanks{Department of Electrical Engineering,
Technion--Israel Institute of Technology, Haifa 32000, Israel}
\thanks{Research supported in part by the ISF (grant 1315/12)}
\and
Subhamay Saha${}^*{}^\dagger$
}

\date\today

\maketitle

\begin{abstract}
We consider the serve-the-longest-queue discipline for a multiclass queue with buffers
of equal size, operating under (i) the conventional and (ii) the Halfin-Whitt
heavy traffic regimes, and show that while the queue length process' scaling limits
are fully determined by the first and second order data in case (i),
they depend on finer properties in case (ii). The proof of the latter relies on
the construction of a {\it deterministic} arrival pattern.
\end{abstract}

\section{Introduction}\label{sec1}

We analyze the multi-class queue in two different diffusion
regimes, namely the {\it conventional} and the {\it Halfin-Whitt} (HW)
heavy traffic regimes, operating under the serve-the-longest-queue (SLQ) scheduling policy.
In both regimes the traffic intensity is asymptotic to unity, where in conventional
heavy traffic, the model is based on a single server
and the arrival rate and service time distributions are scaled up,
while in the HW regime, the arrival rate and
number of servers are scaled up and the service time distributions are kept fixed;
see \cite{ata-nds} and references therein for more on these regimes.
Our goal is to demonstrate that if the buffers are finite and of equal size,
then, perhaps counterintuitively,
the first and second order data of the underlying primitive processes do not uniquely determine
the queue length asymptotics in the HW regime
(the term `first and second order data of the underlying primitive processes'
informally means their Law of Large Numbers and Central Limit Theorem
limit laws; it is rigorously defined in Section \ref{sec2}).
As a result, a diffusion limit does not always exist under the `usual' set of assumptions.
This stands in contrast to the conventional regime where, as we show,
the limit is fully determined by the first and second order data.

Our motivation to study systems with finite buffers stems from
a recent treatment \cite{ata-sah-1}, where they arise in a game-theoretic setting
of customers that act strategically, and avoid joining the queue
if they expect that the delay will exceed a threshold.
In that setting, determining the diffusion-scale asymptotics
of the queue length provides a crucial step in the analysis of a Nash equilibrium.
The usual role played by finite buffers, namely to model finite storage room,
provides, of course, an additional motivation.

In Sections \ref{sec2} and \ref{sec3} we treat the HW and the conventional regimes,
respectively, where in the former we provide a counterexample
to existence of limits, and in the latter we determine the limit.
The aforementioned counterexample is based on the construction
of a certain deterministic arrival pattern; the problem of whether existence of limits
fails under a more common model for arrivals, such as renewal processes
with no fixed times of discontinuity, is left open (see Problem \ref{prob1}).

\skp

We use the following notation. For $a,b\in\R$, the maximum [resp., minimum] is denoted by
$a\vee b$ [resp., $a\w b$], and $a^+=a\vee0$, $a^-=(-a)\vee0$.
For $x,y\in\R^k$ ($k$ a positive integer), $x\cdot y$
and $\|x\|$ denote the usual scalar product and $\ell_2$ norm, respectively.
Write $\{\be_i\}$, $i=1,\ldots,k$ for the standard basis in $\R^k$
and $1$ for $\sum_{i=1}^k\be_i$.
Denote $\R_+=[0,\iy)$, and let $\io:\R_+\to\R_+$ the identity.
For $f:\R_+\to\R^k$, $\|f\|_T=\sup_{t\in[0,T]}\|f(t)\|$,
and, for $\theta>0$,
$w_T(f,\theta)=\sup_{0\le s<u\le s+\theta\le T}\|f_u-f_s\|$.
For a Polish space $\calS$, let $\C_\calS([0,T])$ and $\D_\calS([0,T])$
denote the set of continuous and, respectively, cadlag functions $[0,T]\to\calS$.
Write $\C_\calS$ and $\D_\calS$ for the case where $[0,T]$ is replaced by $\R_+$.
Endow $\D_\calS$ with the Skorohod $J_1$ topology. Write $X_n\To X$ for convergence in
distribution. A sequence of processes $X_n$ with sample paths in
$\D_\calS$ is said to be {\it $C$-tight} if it is tight and every subsequential
limit has, with probability 1, sample paths in $\C_\calS$.
For a positive integer $k$, $m\in\R^k$ and a symmetric, positive matrix $A\in\R^{k\times k}$,
{\it an $(m,A)$-Brownian motion} (BM) is a $k$-dimensional BM starting from zero,
having drift $m$ and infinitesimal covariance matrix $A$.

\section{\textbf{A counterexample to existence of limits in the Halfin-Whitt regime}}\label{sec2}

A sequence of queueing models, indexed by $n\in\N$, and defined on a probability
space $(\Om,\calF,\PP)$,
has $n$ identical servers and a fixed number, $N\ge2$, of buffers
dedicated to customers of $N$ classes.
For $i=1,2,\ldots,N$, class-$i$ customers arrive according to an arrival process $E^n_i$
and upon arrival go directly for service on the event that any of the servers is available,
and otherwise are queued in buffer $i$ if the buffer is not fully occupied. Arrivals are lost
when the corresponding buffer is full.
When a server becomes available and the buffers are non-empty,
it picks a customer from the buffer with most customers,
and, in case of equal maximal queue lengths, a fair
$N$-coin is tossed to determine which buffer to pick from.
Class-$i$ jobs take exponential time to process, with parameter $\mu_i^n$, where
\begin{equation}\label{10}
\mu^n_i=\mu_i+n^{-1/2}\hat\mu_i+o(n^{-1/2}),
\end{equation}
and $\mu_i>0$ and $\hat\mu_i\in\R$ are constants.
The arrival counting processes, $E^n_i$, are assumed to satisfy the Law of Large Numbers,
\begin{equation}\label{12}
\bar E^n_i:=n^{-1}E^n_i\To\la_i\io,
\end{equation}
where $\la_i>0$ are constants, and the Central Limit Theorem,
\begin{equation}\label{13}
\hat E^n_i:=n^{-1/2}(E^n_i-\la^n_i\io)\To W_i^{\text{arr}},
\end{equation}
where $\la^n_i=\la_i n+n^{1/2}\hat\la_i+o(n^{1/2})$, and $W_i^{\text{arr}}$ is a $(0,\lambda_i\sig_i^2)$-BM,
for constants $\hat\la_i\in\R$, $\sig^2_i\ge0$.
It is also assumed that arrival processes are independent.
The resulting asymptotic
traffic intensity is given by $\sum_i\rho_i$, where $\rho_i=\la_i/\mu_i$,
assumed to satisfy the critical load condition
$\sum_i \rho_i=1$.
The queue length processes are denoted by $Q^n=(Q^n_1,\ldots,Q^n_N)$.
The number of class-$i$ customers in the system (resp., in the buffer, in service)
at time $t$ is denoted by $X^n_i(t)$ (resp., $Q^n_i(t)$, $\Ps^n_i(t)$).
Note that $X^n=Q^n+\Ps^n$, and $1\cdot\Ps^n\le n$.
Diffusion scaled versions of these processes are denoted by
\[
\hat X^n=n^{-1/2}(X^n-n\rho),\qquad \hat Q^n=n^{-1/2}Q^n,\qquad \hat\Ps^n=n^{-1/2}(\Ps^n-n\rho).
\]
It is assumed that the initial condition $\hat X^n(0)$ satisfies
\[
\hat X^n(0)\To X_0,
\]
where $X_0$ is an $\R^N$-valued r.v., whose distribution is denoted by $m_0$,
and, for simplicity, the queue lengths are assumed to start at zero,
that is, $Q^n(0)=(Q^n_1(0),\ldots,Q^n_N(0))=0$.
We will assume that the buffer sizes, denoted throughout by $\{\beta^n_i\}$,
are asymptotic to $\{\beta_in^{1/2}\}$, where $\beta_i>0$ are constants, namely
$\beta^n_i=\beta_in^{1/2}+o(n^{1/2})$.

The tuples $(\mu_i,\la_i)$ and $(\hat\mu_i,\hat\la_i,\sig_i^2,m_0)$ are often
referred to as {\it first} and {\it second} order data, respectively.
We denote them jointly by
\[
\bdel=(\mu_i,\la_i,\hat\mu_i,\hat\la_i,\sig_i^2,m_0).
\]

Given $k\in\{1,\ldots,N\}$,
consider a stochastic differential equation (SDE) with reflection, for a process $X$ that
lives in
\[
G_k=\{x\in\R^N:1\cdot x\le N\beta_k\},
\]
and reflects on the boundary of $G_k$ in the direction $-\be_k$.
Let $\{W(t)\}$ be a $(\hat\la,A)$-BM, where
$A=\diag(\lambda_i(\sig_i^2+1))$.
Let $b:\mathbb{R}^N\rightarrow \mathbb{R}^N$ be given by
\begin{equation}\label{65}
b(x)=-(\mu_1(x_1-N^{-1}(1\cdot x)^+),\ldots,\mu_N(x_N-N^{-1}(1\cdot x)^+)).
\end{equation}
Let $(X,L)=(X^{(k)},L^{(k)})$ be the unique pair of processes that is adapted to the filtration
$\sig\{X_0\}\vee\sig\{W(u),u\le t\}$, where $X$ has sample paths in $\mathbb{C}(\R_+:G_k)$,
$L$ has nondecreasing sample paths in $\mathbb{C}(\R_+:\R_+)$,
and the pair satisfies a.s.,
\begin{equation}\label{61}
\begin{split}
&X(t)=X_0+W(t)+\int_0^tb(X(u))du-L(t)\be_k,\qquad t\ge0,\\
&\int_{[0,\iy)}1_{\{1\cdot X(t)<N\beta_k\}}dL(t)=0\,.
\end{split}
\end{equation}
The existence and uniqueness of such a pair follows from Proposition 3 of
\cite{andore} on noting that $b$ is Lipschitz continuous.
We denote by $X^{(k)}$ the solution to the SDE \eqref{61}.

It follows from the results of \cite{ata-sah-1} that
the limits of $(\hat X^n,\hat Q^n,\hat\Ps^n)$ are not uniquely determined by $\bdel$
when the buffer sizes are asymptotically equal, i.e., $\beta_i=\beta_1$ for all $i$.
More precisely, the following result appears in \cite{ata-sah-1} (Proposition 4.3):

{\it Assume that for some $k$ and all $i\ne k$, $\beta_k<\beta_i$.
Then $(\hat X^n,\hat Q^n,\hat\Ps^n)\To(X,Q,\Ps)$, where $X=X^{(k)}$ is the unique solution
of \eqref{61}, and $Q$ and $\Ps$ are recovered from it via $Q=N^{-1}(1\cdot X)^+$ and $\Ps=X-Q$.
}

One can draw from this result the following conclusions
regarding the case $\beta_i=\beta_1$ for all $i\in\{1,\ldots,N\}$:
\begin{itemize}
  \item[(i)] For every $k$, one can choose $\{\beta^n_i\}$ asymptotic to $\{\beta_in^{1/2}\}$,
  in such a way that
  $\hat X^n\To X$, where $X=X^{(k)}$. {\it Thus the first and second order data do not determine
  the limits.}
  \item[(ii)] One can choose $\{\beta^n_i\}$ asymptotic to $\{\beta_in^{1/2}\}$
  in such a way that
  $\hat X^n$ do not converge in distribution. {\it Thus limits need not exist.}
\end{itemize}
Indeed, (i) follows because, given $k$, we have $\hat X^n\To X^{(k)}$
when $\beta_k=c-\eps$ and $\beta_i=c$ for all $i\ne k$,
with $c,\eps>0$ fixed; hence by a diagonal argument, the same is true with $\eps$
replaced by $\eps_n>0$, for some $\eps_n\to0$.
Of course, (ii) is immediate from (i).

In the present note, we are interested in the case where $\beta^n_i$
are {\it exactly equal to each other}, for every $n$.
Assuming in what follows that for a constant $\beta_1>0$,
\begin{equation}
  \label{30}
  \beta^n_i=\beta^n:=\lfloor\beta_1n^{1/2}\rfloor,\qquad i\in\{1,\ldots,N\},\, n\in\N,
\end{equation}
we ask whether, in this situation,
the first and second order data still fall short of determining the limit behavior.
More precisely, we aim at addressing the following assertions:
\begin{itemize}
  \item[(i')] For every $k$ one can choose $\{E^n_i\}$ that satisfy \eqref{12}
  and \eqref{13}, in such a way that $\hat X^n\To X$, where $X=X^{(k)}$.
  \item[(ii')] One can choose $\{E^n_i\}$ that satisfy \eqref{12} and \eqref{13},
  in such a way that $\hat X^n$ do not converge.
\end{itemize}
An affirmative answer will confirm that the first and second order data do not determine the
limits even when the buffers are exactly equal in size. We address these questions in the special case
where $N=2$, but it will be clear from the proof that analogous treatment
is possible in general.
\begin{theorem}
  \label{th1}
  Consider $N=2$ and assume that the buffer sizes are given by \eqref{30}.
  Fix $k\in\{1,2\}$.
  Then one can find $\{E^n_1\}$ and $\{E^n_2\}$ satisfying \eqref{12} and \eqref{13},
  so that $\hat X^n\To X^{(k)}$, the solution of the SDE \eqref{61}.
\end{theorem}

Note that the domain $G_k$ does not depend on $k$ in this case, since $\beta_1=\beta_2$.
However, the SDEs still differ in terms in the direction of reflection,
and in this situation the solutions $X^{(1)}$ and $X^{(2)}$ are not equal in law.
Hence the validity of (i') and (ii') is an immediate consequence of the above result.

As mentioned earlier,
the proof of the result will be based on the construction of a deterministic
arrival pattern. It is natural to ask
whether the result remains valid under the additional requirement that
the arrivals follow a more common model, such as renewals. More precisely,
we formulate the following problem, that we leave open.
\begin{problem}
  \label{prob1}
  Determine whether existence of limits may fail when the arrivals are given
  by accelerated versions of independent renewal processes (namely, $E^n_i(t)=E_i(\mu^n_it)$,
  $t\ge0$) with inter-renewal distributions that have density.
\end{problem}

%\skp

\noi{\bf Proof of Theorem \ref{th1}:}
The construction will be with the parameters $\la_i=1$, $\hat\la_i=0$, $\mu_i=2$,
$\hat\mu_i=0$, $\rho_i=1/2$.
The arrival processes we construct are deterministic, and satisfy \eqref{12},
as well as \eqref{13} with $\sig_i=0$. In particular, the driving BM in \eqref{61}
is a $(0,A)$-BM where $A=\diag(\la_i)=\diag(1,1)$.
The construction is presented for $k=1$; the case $k=2$ is obtained by interchanging the roles
of class 1 and class 2.

Fix a sequence $m_n=\lfloor n^a\rfloor$, $n\in\N$, where $a\in(0,\frac{1}{2})$ is constant.
For ease of notation we suppress the index $n$ in $m_n$ and $\beta^n$ (of \eqref{30})
and write $m$ and $\beta$, respectively.

First, we construct $E^n_i$ on the interval $[0,\tau]$, where $\tau=\frac{m}{n}$,
by letting
\begin{align}
\notag
E^n_1(t)&=\begin{cases}
0,& t\in[0,\tau),\\
m,& t=\tau,
\end{cases}
\\
\label{40}
E^n_2(t)&=\begin{cases}
  0, & t\in[0,\frac{\tau}{2}),\\
  \lfloor2n(t-\frac{\tau}{2})+2\rfloor, & t\in[\frac{\tau}{2},\tau),\\
  m, & t=\tau.
\end{cases}
\end{align}
Thus, for each class, $m$ arrivals occur during $[0,\tau]$,
where class-1 customers all arrive at time $\tau$, whereas class-2 arrivals
are at $\frac{\tau}{2}, \frac{\tau}{2}+\frac{1}{2n},\frac{\tau}{2}+\frac{2}{2n},
\ldots,\frac{\tau}{2}+\frac{m-1}{2n}$.
Beyond $[0,\tau]$, the pattern defined on $(0,\tau]$ repeats itself with period $\tau$.
Namely, $E^n$ is given by
\[
E^n(t+j\tau)=E^n(j\tau)+E^n(t),\qquad t\in(0,\tau],\,j\in\N.
\]
Note that for both $i=1,2$, \eqref{12} holds with $\la_i=1$, and \eqref{13}
holds with $\la^n_i=n$ and $\sig^2_i=0$ (thus $W^{\rm arr}_i=0$ a.s.).
The parameters $\mu^n_i$ are given by $\mu^n_i=2$ for $n\in\N$, $i=1,2$.

We need some additional notation. Denote by $R^n_i$ the counting process for class-$i$
losses since time 0,
by $B^n_i$ the counting process for class-$i$ customers sent to the service pool
since time 0, by $D^n_i$ the counting process for class-$i$ departures from service,
and by $S$ a unit-rate Poisson process representing potential service.
Namely,
\begin{equation}\label{20}
1\cdot D^n(t)=S\Big(2\int_0^t1\cdot\Ps^n(u)du\Big)\,.
\end{equation}
We have the following balance equations
\begin{align}\label{50}
Q^n_i(t)=Q^n_i(0)+E^n_i(t)-B^n_i(t)-R^n_i(t)\,,
\end{align}
\begin{align}\label{19}
\Ps_i^n(t)=\Ps_i^n(0)+B^n_i(t)-D_i^n(t)\,.
\end{align}
Denote $\hat R^n=n^{-1/2}R^n$.
The main estimate will be to show that $\hat R^n_2\To0$.
Fix $T$ and note that
\begin{equation}\label{44}
\EE R^n_2(T)\le\sum_{j=0}^{\lfloor T/\tau\rfloor}\EE[R^n_2(j\tau+\tau)-R^n_2(j\tau)]
\le\sum_{j=0}^{\lfloor T/\tau\rfloor}m\PP[\sup_{t\in[j\tau+\frac{\tau}{2},j\tau+\tau)}Q^n_2(t)=\beta],
\end{equation}
where we used the fact that a class-2 loss can only occur if a customer arrives
when the buffer is full (that is, $Q^n_2=\beta$), that class-$2$
arrivals occur only within $[j\tau+\frac{\tau}{2},j\tau+\tau)$,
and that the total number of losses over each such interval is bounded by $m$.

Towards bounding the RHS of \eqref{44},
note that, by construction, for each $n$, the tuple
$\Sig^n:=(Q^n_1,Q^n_2,\Ps^n)$ forms an inhomogeneous Markov process on the state space
\[
\calS^n:=\{(q_1,q_2,\psi)\in\Z_+^3:q_1\vee q_2\le\beta,\psi\le n, (q_1+q_2)\w(n-\psi)=0\},
\]
where the first constraint expresses the buffer limit, the second states that the
number of jobs in service does not exceed the number of servers, and the last
corresponds to the non-idling condition (the inhomogeneity is due to the structure
of arrivals).
Denote by $\PP^n_x$, $x\in\calS^n$, the corresponding Markov family,
where $x$ serves as the initial condition, i.e., $\PP^n_x(\Sig^n(0)=x)=1$.
Although $\Sig^n$ is not a homogeneous Markov process,
The path-valued Markov chain $\{\Sig^n|_{(j\tau,j\tau+\tau]}\}$, $j\in\N$ is homogeneous by
construction, and in particular,
\begin{equation}\label{45}
\PP[\sup_{t\in[j\tau+\frac{\tau}{2},j\tau+\tau)}Q^n_2(t)=\beta|\Sig^n(j\tau)=x]
=\PP^n_x(\sup_{t\in[\frac{\tau}{2},\tau)}Q^n_2(t)=\beta).
\end{equation}
Below, we show that
\begin{equation}\label{31}
\sup_{x\in\calS^n}\PP^n_x(\sup_{t\in[\frac{\tau}{2},\tau)}Q^n_2(t)=\beta)\le c_1e^{-c_2m},
\end{equation}
where $c_1,c_2>0$ are constants that do not depend on $n$ or $k$.
(Note that the initial condition $x$ could have $q_2=\beta$, but this
does not contradict \eqref{31} which is a statement regarding
the times $[\tau/2,\tau)$.)
Combining \eqref{45} with the estimates \eqref{44} and \eqref{31} gives
\[
\EE R^n_2(T)\le c_1\frac{T}{\tau}me^{-c_2m}=c_1Tne^{-c_2m}.
\]
Recalling that $m=\lfloor n^a\rfloor$, where $a>0$, gives $R^n_2(T)\To0$ as $n\to\iy$,
and therefore $\hat R^n_2(T)\To0$.

The intuitive explanation of \eqref{31} is simple. During the first half of the period,
$(0,\tau/2)$, there are no arrivals, and both queue lengths drop dramatically below the
buffer size $\beta$, regardless of their initial condition. On $[\tau/2,\tau)$,
there are still no class-1 arrivals,
and so if $Q^n_2$ comes near $\beta$, it is necessarily the longer among the two queues.
At these times, class-2 jobs receive all service effort, which again causes $Q^n_2$ to drop.

To prove \eqref{31}, fix $x\in\calS^n$.
Denote $\theta=\inf\{t\in[\tau/2,\tau):Q^n_2(t)=\beta\}$.
The event indicated in \eqref{31} can be written as $\{\theta<\iy\}$
(equivalently, $\{\theta\le\tau\}$).
Note first that on that event, it is impossible to have $1\cdot\Ps^n(s)<n$ for some
$s\in[0,\theta]$, when $n$ is sufficiently large. Namely, if
$n$ is large then $m=m_n=\lfloor n^a\rfloor<\beta=\beta^n=\lfloor\beta_1n^{1/2}\rfloor$. Note that
non-idling condition can be expressed as
\[
\text{for every $t$, } 1\cdot Q^n(t)>0 \text{ implies } 1\cdot\Ps^n(t)=n.
\]
Hence the existence of such $s$ implies $Q^n_2(s)=0$, and thus by \eqref{50},
\[
\beta-0=Q^n_2(\theta)-Q^n_2(s)\le E^n_2(\theta)-E^n_2(s)\le m,
\]
that contradicts $m<\beta$. As a result, using also \eqref{20},
on the event $\{\theta<\iy\}$, one has
\begin{equation}\label{46}
1\cdot D^n(t)=S(2nt),\qquad t\le\theta.
\end{equation}

Next, on the time interval $[0,\tau]$,
all class-$1$ arrivals occur at time $\tau$, thus if there are
any losses at this class, they also occur at that time. Thus, by \eqref{50},
\begin{equation}\label{47}
Q^n_1(t)=Q^n_1(0)-B^n_1(t),\qquad t\in[0,\tau).
\end{equation}
As for $Q^n_2$, the same is true regarding the interval $[0,\frac{\tau}{2})$. Thus
\[
Q^n_2(t)=Q^n_2(0)-B^n_2(t),\qquad t\in[0,\frac{\tau}{2}).
\]
Hence
\[
1\cdot Q^n(\frac{\tau}{2}-)=1\cdot Q^n(0)-1\cdot B^n(\frac{\tau}{2}-)
=1\cdot Q^n(0)-S(m),
\]
where we used \eqref{19} and \eqref{46}.
Now, using the fact that each queue length is bounded above
by $\beta$, it follows from the property of the policy to always offer
service to the longer queue
that, for any $\ell\in\N$, once $2\ell$ jobs are removed from the buffers and sent
to service, each of the queue lengths is bounded above by
$\beta-\ell$. As we have just argued, on the event $\{\theta<\iy\}$
there are $S(m)$ such removals during $[0,\tau/2)$, hence
\[
Q^n_1(\frac{\tau}{2}-)\vee Q^n_2(\frac{\tau}{2}-)\le\beta-\Big\lfloor\frac{S(m)}{2}\Big\rfloor.
\]
If indeed $\theta<\iy$, namely, $Q^n_2$ reaches $\beta$ during $[\frac{\tau}{2},\tau)$,
then there must exist a time $u\in[\tau/2,\theta]$ such that
\[
Q^n_2(u-)=\beta-\Big\lfloor\frac{S(m)}{2}\Big\rfloor,
\qquad \beta-\Big\lfloor\frac{S(m)}{2}\Big\rfloor<Q^n_2(t)<\beta,\, t\in[u,\theta).
\]
Using \eqref{50}, noting there are no losses on this interval,
\[
\Big\lfloor\frac{S(m)}{2}\Big\rfloor
=Q^n_2(\theta)-Q^n_2(u-)=E^n_2(\theta)-E^n_2(u-)-B^n_2(\theta)+B^n_2(u-).
\]
Also $Q^n_2>Q^n_1$ must hold on the interval $[u,\theta)$,
since by \eqref{47}, $Q^n_1$ can only decrease from $Q^n_1(\frac{\tau}{2}-)$.
Thus the increment of $B^n_2$ equals that of $1\cdot B^n$. In turn, using \eqref{19}
and the fact that $1\cdot\Ps^n=n$ on this interval, this increment
is equal to the increment of
$1\cdot D^n$, which, by \eqref{46} is given by $S(2n\theta)-S(2nu-)$.
We thus obtain
\[
\Big\lfloor\frac{S(m)}{2}\Big\rfloor=E^n_2(\theta)-E^n_2(u-)-S(2n\theta)+S(2nu-).
\]
By \eqref{40},
\[
|E^n_2(\theta)-E^n_2(u-)-2n(\theta-u)|\le 3.
\]
Recalling that, on $\theta<\iy$, $u,\theta\in[\tau/2,\tau]$, it follows that
\[
\PP^n_x(\theta<\iy)\le\PP\Big(\sup_{s,t\in[m,2m]}|S(t)-S(s)-(t-s)|\ge \frac{S(m)}{2}-5\Big).
\]
Denoting $\bar S^m(t)=\frac{S(mt)-mt}{m}$, we have
\begin{align*}
\PP^n_x(\theta<\iy) &\le \PP\Big(\frac{S(m)}{2}-5<\frac{m}{4}\Big)
+\PP\Big(\sup_{s,t\in[1,2]}|\bar S^m(t)-\bar S^m(s)|\ge\frac{1}{4}\Big)\\
&\le\PP\Big(\bar S^m(1)<-\frac{1}{2}+\frac{10}{m}\Big)
+\PP\Big(\sup_{t\in[0,2]}|\bar S^m(t)|\ge\frac{1}{8}\Big).
\end{align*}
Note that the expression on the RHS does not depend on $x$. Moreover, by the
sample path large deviations principle satisfied by $\bar S^m$, each
of the two terms above is bounded by $c_1e^{-c_2m}$, for constants $c_1,c_2>0$ that do not depend
on $m$. This completes the proof of \eqref{31}. As we have argued above, this gives $\hat R^n_2\To0$.

Based on the above, the completion of the proof follows closely along the lines of
Section 4 of \cite{ata-sah-1}. Thus, for this part, we only provide a sketch.
First, the model \eqref{20} for departures, based on the primitive data $S$, can alternatively
be represented in terms of a pair of potential service processes, namely two rate-1
Poisson processes $S_1$ and $S_2$, that are mutually independent, and independent of
the system's initial condition:
\[
D^n_i(t)=S_i\Big(\mu_i\int_0^t\Ps^n_i(u)du\Big).
\]
Next, the balance equations \eqref{50} and \eqref{19} translate to the diffusion scale as
\begin{align}\label{41}
&\hat{Q}^n_i(t)= \hat{Q}^n_i(0)+\hat{E}^{n}_i(t)-\hat{B}^n_i(t)-\hat{R}^n_i(t)\,,\\ \label{42}
&\hat{\Ps}^n_i(t)=\hat{\Ps}^n_i(0)+\hat{B}^n_i(t)
-\hat{S}^n_i\Big(\mu_i\int_0^t\bar{\Ps}^n_i(u)du\Big)-\mu_i\int_0^t\hat{\Ps}^n_i(u)du\,,
\end{align}
where
\[
\bar\Ps^n_i=n^{-1}\Ps^n_i,
\qquad
\hat S^n_i=n^{-1/2}(S_i(n\io)-n\io),
\qquad
\hat B^n_i=n^{-1/2}(B^n_i-n\la_i\io).
\]
Hence
\begin{align}\label{43}
\hat{X}^n_i=\hat{Q}^n_i+\hat\Ps^n_i
=\hat X^n_i(0)+\hat W^n_i-\mu_i\int_0^\cdot(\hat X^n_i(u)-\hat Q^n_i(u))du-\hat R^n_i,
\end{align}
where
\begin{align}\label{51}
\hat{W}_i^n=\hat{E}^n_i-\hat{S}_i^n\Big(\mu_i\int_0^\cdot\bar{\Ps}_i^n(u)du\Big).
\end{align}
Fix a sequence $k_n$, $n\in\N$, such that
$\lim n^{-1/2}k_n=\iy$ and $\lim n^{-1}k_n=0$, and, given $T<\iy$, define
$T_n=\inf\{t:1\cdot R^n(t)\ge k_n\}\w T$.
Lemma 4.2 of \cite{ata-sah-1} states that, for $i=1,2$,
$\|\hat{Q}^n_i-N^{-1}(1\cdot\hat{X}^n)^+\|_{T_n}\to 0$,
and $\|\bar{\Ps}^n_i(t)-\rho_i\|_{T_n}\to 0$, in probability, as $n\to\iy$.
In the proof of Proposition 4.3 of \cite{ata-sah-1} it is shown that
$\PP(T_n<T)\to0$ as $n\to\iy$. As a result, in the above two statements, $T_n$
can be replaced by $T$, namely, for any $T<\iy$, for $i=1,2$,
\begin{equation}
  \label{101}
\|\hat{Q}^n_i-N^{-1}(1\cdot\hat{X}^n)^+\|_T\to 0,
\qquad
\|\bar{\Ps}^n_i-\rho_i\|_T\to 0,
\qquad
\text{in probability, as } n\to\iy.
\end{equation}
By the central limit theorem, ($\hat S^n_1,\hat S^n_2)\To W$, where $W$ is
a $(0,A)$-BM, with $A=\diag(1,1)$. Since $\mu_i=2$ and $\rho_i=1/2$,
it follows that $\hat W^n\To W$.

Define $\Gam:\mathbb{D}_{\mathbb{R}^2}([0,T])\rightarrow \mathbb{D}_{\mathbb{R}^2}([0,T])$ by
\begin{equation}\label{100}
\Gam(f)(t)=f(t)-g(t)\be_1\,,
\qquad
g(t)=\sup_{0\leq u\leq t}(2\beta-1\cdot f(u))^-\,.
\end{equation}
The following two properties follow directly from the definition, namely
there exists a constant $C$ such that
\begin{equation}
  \label{60}
  \|\Gam(f)-\Gam(\tilde f)\|_T\le C\|f-\tilde f\|_T,\qquad f,\tilde f\in\mathbb{D}_{\R^2}([0,T]),
\end{equation}
and
\begin{equation}
  \label{62}
  w_T(\Gam(f),\cdot)\le Cw_T(f,\cdot),\qquad f\in\mathbb{D}_{\R^2}([0,T]).
\end{equation}
Given $z\in\mathbb{D}_{\R^2}$, $z(0)\in G:=\{x\in\R^2:1\cdot x\le 2\beta\}$, we say that
$(y,\ell)\in\mathbb{D}_{\R^2}\times
\mathbb{D}_{\R}$ solves the Skorohod problem (SP)
in $G$, with reflection in the direction $-\be_1$,
for data $z$, if $y(t)\in G$ for all $t$, $\ell$ is nonnegative and nondecreasing, and
\[
y=z-\ell\be_k,\qquad\int_{[0,\iy)}1_{\{1\cdot y<2\beta\}}d\ell=0.
\]
It is well known that for $z$ as above, a necessary and sufficient condition for $(y,\ell)$
to be a solution is that $y=\Gam(z)$.

Based on the fact that $\hat R^n_2\To0$ and \eqref{101}, there exists a process
$\tilde X^n$ such that $\tilde X^n-\hat X^n\To0$,
$\tilde X^n(t)\in G$ for all $t$, and
\[
\tilde X^n=\hat X^n(0)+\hat W^n+\int_0^\cdot b(\tilde X^n(u))du-\hat{R}^n_1\be_1+\eps^n,
\]
where $\eps^n$ is a sequence of processes converging to $0$ in probability, and
$\int 1_{\{1\cdot\tilde X^n < 2\beta\}}d\hat{R}^n_1 = 0$.
As a result,
\begin{align}
\tilde X^n={\Gam}\Big(\hat X^n(0)+\hat W^n
+\int_0^{\cdot}{b}(\tilde X^n(u))du+\eps^n\Big).
\end{align}
Taking limits, using properties \eqref{60} and \eqref{62} gives the convergence result.
\qed

\section{A limit result in conventional heavy traffic}\label{sec3}

In this section we show that in conventional heavy traffic, the first and second order
data of the primitives fully determine the diffusion-scale behavior, and in particular,
the diffusion limit exists. The purpose of presenting this result is mainly to contrast it with
the previous section's counterexample.
An important distinction between the two regimes is that the HW regime gives rise to
a nondegenerate $N$-dimensional diffusion process (such as \eqref{61}), whereas
in the conventional regime the limit is a 1-dimensional diffusion.
It therefore comes as no surprise that the reflection due to the buffer size constraint
can only occur according to the 1-dimensional Skorohod map.
While the result appears to be standard, we have not been able to find it in the literature.

The model is similar to the one considered in Section \ref{sec2}, but has only
one sever. The probabilistic assumptions regarding arrivals are as before,
namely they satisfy \eqref{12} and \eqref{13}.
The service time distribution is general.
With $S^n_i$ denoting the potential service counting process for class-$i$ customers,
it is assumed, analogously to \eqref{12} and \eqref{13}, that $n^{-1}S^n_i\To\mu_i\io$,
and $\hat S^n_i:=n^{-1/2}(S^n_i-n\mu_i\io)\To W^{\rm ser}_i$, where
$W^{\rm ser}_i$ is a $(0,\mu_i\gamma_i^2)$-BM, and $\mu_i>0$, $\gamma_i\ge0$ are constants.
For each $n$, the $2N$ processes $(A^n_i,S^n_i)$ are mutually independent.

As before, the sequence of queueing networks approaches heavy traffic,
i.e., the limiting traffic intensity $\sum \rho_i=1$, where $\rho_i=\lambda_i/\mu_i$,
the scheduling is according to SLQ, and server is non-idling.
We also assume that the system is initially empty.
The number of class-$i$ customers in the system at time $t$ is denoted by $X^n_i(t)$.
If $T^n_i(t)$ is the service time devoted to class-$i$ customers up to time $t$
and $R^n_i(t)$ counts the number of lost arrivals up to time $t$ then we have
\begin{align}
X^n_i(t)= E^n_i(t)-S^n_i(T^n_i(t))- R^n_i(t)\,.
\end{align}
The $i$th buffer size is given by $\beta^n_i=\beta n^{1/2}+\eps^n_in^{1/2}$,
where $\eps^n_i\to0$ for each $i$, and $\beta>0$ is a constant.
Denote the diffusion-scale versions of the processes by
$\hat{X}^n_i=n^{-1/2}X^n_i$ and $\hat{R}^n_i=n^{-1/2}R^n_i$.
Straightforward calculation gives
\begin{align}\label{systeq}
\hat{X}^n_i=\hat{W}^n_i+\hat{Y}^n_i-\hat{R}^n_i,
\end{align}
where
\begin{equation}\label{96}
\hat{W}^n_i(t)=\hat{E}^n_i(t)-\hat{S}^n_i(T^n_i(t))+\hat\lambda^n_it\,,
\qquad
\hat{Y}^n_i=\mu_in^{1/2}(\rho_i \io - T^n_i)\,,
\end{equation}
and $\hat\la^n_i:= (\la^n_i-n\lambda_i)n^{-1/2} \to\hat\la_i$, by the assumption
made following equation \eqref{13}.
The following is often referred to as a {\it state space collapse} result.
\begin{lemma}\label{lem1}
The scaled number of customers in the various classes are asymptotically equal. Namely,
$\max_{i,j}\|\hat{X}^n_i-\hat{X}^n_j\|_T\Rightarrow 0$, for any $T < \infty$.
\end{lemma}
\begin{proof}
The proof follows along the lines of Proposition 1 in \cite{mieghem}, with minor modifications for finite buffers.
\qed
\end{proof}

\skp

For $a>0$, the Skorohod map on the interval $[0,a]$ will be
denoted by $\Gam_{[0,a]}$. It maps $\D_\R$ to itself,
and is characterized as the first component of
the solution map $\psi\to(\ph,\eta_1,\eta_2)$ to the
problem of finding, for a given $\psi$, a triplet $(\ph,\eta_1,\eta_2)$, such that
\[
\begin{split}
&\ph=\psi+\eta_1-\eta_2,\qquad \ph(t)\in[0,a] \text{ for all } t,\\
&\text{$\eta_i$ are nonnegative and nondecreasing, $\eta_i(0-)=0$,
and}
\\
&\text{$\int_{[0,\iy)}1_{(0,a]}(\ph)d\eta_1=\int_{[0,\iy)}1_{[0,a)}(\ph)d\eta_2=0$.}
\end{split}
\]
Existence and uniqueness of solutions are well-known (see eg.\ \cite{KLRS}).

Denote $\alpha=\big(\sum_{i=1}^N\mu_i^{-1}\big)^{-1}$.
Let $\tilde W$ be a (one-dimensional) $(\tilde m,\tilde A)$-BM, where
$\tilde m=\al\sum_{i=1}^N\frac{\hat{\lambda}_i}{\mu_i}$ and
$\tilde A=\al^2\sum_{i=1}^N\frac{\lambda_i}{\mu_i^2}(\sigma^2_i+\gamma_i^2)$.
Then the process $\tilde X:=\Gam|_{[0,\beta]}(\tilde W)$ is a reflected BM on $[0,\beta]$.

\begin{theorem}
We have $(\hat{X}^n_1, \ldots, \hat{X}^n_N)\Rightarrow (\tilde X,\ldots,\tilde X)$.
\end{theorem}

\begin{proof}
Define $\tilde{X}^n=\al\sum_{i=1}^N\mu_i^{-1}\hat{X}^n_i$.
It follows from Lemma \ref{lem1} that there exists a sequence $\bar{\delta}_n\rightarrow 0$,
such that, with
$$
\Om_n=\{\max_i\|\hat{X}^n_i-\tilde{X}^n\|_T< \bar{\delta}_n\},
$$
one has $\mathbb{P}(\Om_n)\rightarrow 1$ as $n\rightarrow \infty$. Now, by \eqref{systeq},
$\tilde{X}^n=\tilde{W}^n+\tilde{Y}^n-\tilde{R}^n$,
where
\[
\tilde{W}^n=\al\sum_{i=1}^N \frac{\hat{W}^n_i}{\mu_i},
\qquad
\tilde{Y}^n=\al n^{1/2}\Big(\io-\sum_{i=1}^N T^n_i\Big),
\qquad
\tilde{R}^n=\al\sum_{i=1}^N \frac{\hat{R}^n_i}{\mu_i}.
\]
Note that $t-\sum_{i=1}^NT_i^n(t)$ gives the cumulative idleness time of the server by time
$t$. As a result, the process $\tilde{Y}^n$ is non-decreasing, $\tilde{Y}^n(0)=0$,
and by the non-idling condition, increases only when $\tilde{X}^n=0$.
Moreover, $\tilde{R}^n$, is non-decreasing, starts from $0$ and since arriving jobs
are lost only when the corresponding buffer is full, this process increases only when
$\max_i\hat{X}^n_i\geq \beta-\delta_n$, where we denote $\delta_n=\max_i|\eps^n_i|$.
As a result, on the event $\Om_n$,
$\tilde{R}^n$ increases only when $\tilde{X}^n \geq a_n:=\beta-\delta_n- \bar{\delta}_n$.
On $\Om_n$ we have
\begin{align*}
\tilde{X}^n(t)&=\alpha\sum_{i=1}^N\frac{\hat{X}^n_i(t)}{\mu_i}
\leq \alpha\frac{1}{\alpha}(\beta+\delta_n)=\beta+\delta_n\,.
\end{align*}
Defining $X^{*,n}=\tilde{X}^n\wedge a_n$,
we have $X^{*,n}=\tilde{X}^n+e^n_1$, where $e^n_1$ is a process that
satisfies $|e^n_1(t)|\leq 2\delta_n+\bar{\delta}_n$ for all $t$, on $\Om_n$.
Since $\PP(\Om_n)\to1$, $e^n_1$ converges to zero in probability.
By the discussion above, we also have on $\Om_n$,
\[
X^{*,n}=e^n_1+\tilde{W}^n+\tilde{Y}^n-\tilde{R}^n,
\qquad
X^{*,n}(t)\in[0,a_n] \text{ for all } t,
\]
\[
\int_{[0,\iy)}1_{(0,a_n]}(X^{*,n})d\tilde Y^n=\int_{[0,\iy)}1_{[0,a_n)}(X^{*,n})d\tilde R^n=0.
\]
As a result, $X^{*,n}=\Gamma_{[0,a_n]}(e^n_1+\tilde{W}^n)$ on $\Om_n$.
It follows from the explicit expression for the Skorohod map, provided in \cite{KLRS},
that $\|\Gam_{[0,a_1]}(\psi)-\Gam_{[0,a_2]}(\psi)\|_T\le a_2-a_1$, for any
$T<\iy$, $0<a_1<a_2<\iy$ and $\psi$. As a result,
$X^{*,n}=\Gamma_{[0,\beta]}(e^n_1+\tilde{W}^n)+e^n_2$, holds on $\Om_n$,
where $\|e^n_2\|_T\le\del_n+\bar\del_n$. Hence, on all of $\Om$,
\begin{equation}\label{97}
\tilde X^n=\Gamma_{[0,\beta]}(e^n_1+\tilde{W}^n)+e^n_3,
\end{equation}
where $e^n_3$ converges to zero in probability.
By \eqref{96} and the assumed convergence of the processes $\hat E^n_i$, $\hat S^n_i$
and constants $\hat\la^n_i$, it follows that $\tilde{W}^n$
is a $C$-tight sequence of processes. As a result of relation \eqref{97} and
the continuity of $\Gam_{[0,\beta]}$ as a map from $\D_\R([0,T])$
(for arbitrary $T$), equipped with the uniform topology, to itself,
$(\tilde X^n,\tilde{Y}^n,\tilde{R}^n)$ is also a $C$-tight sequence.
Hence we obtain from \eqref{96} that $T^n_i\Rightarrow \rho_i \io$.
It follows that $\tilde W^n\To\tilde W$. Arguing again by
the continuity of the Skorohod map, we obtain
$\tilde X^n\Rightarrow \Gamma_{[0,\beta]}(\tilde{W})$. The result now follows.
\qed
\end{proof}

\skp

\noi{\bf Acknowledgment.} The authors are grateful to the two referees for their valuable comments.

\footnotesize

\bibliographystyle{is-abbrv}

%\bibliographystyle{is-alpha}

%\bibliography{../refs4bibtex/refs}

%\bibliography{refs}

\end{document}